\documentclass[11pt,reqno]{amsart}
\usepackage[centertags]{amsmath}
\usepackage{mathrsfs}

\theoremstyle{plain} 

\newtheorem{pro}{Proposition}
\newtheorem{thm}{Theorem}

\theoremstyle{definition}

\theoremstyle{remark}

\numberwithin{equation}{section}

\DeclareMathSymbol{\R}{\mathalpha}{AMSb}{"52}
\DeclareMathSymbol{\C}{\mathalpha}{AMSb}{"43}
\newcommand{\mbb}[1]{\mathbb{#1}}
\newcommand{\Z}{\mbb{Z}}

\newcommand{\comment}[1]{}

\newcommand{\bd}{\begin{description}}
\newcommand{\ed}{\end{description}}
\newcommand{\beqr}{\begin{eqnarray}}
\newcommand{\eeqr}{\end{eqnarray}}
\newcommand{\beqt}{\begin{equation}}
\newcommand{\eeqt}{\end{equation}}

\begin{document}

\title[]{A characterization of  iterative equations by their coefficients}

\author[JC Ndogmo]{J.C. Ndogmo }
\address[JC Ndogmo]{School of Mathematics\\
University of the
Witwatersrand\\
Private Bag 3, Wits 2050\\
South Africa}
\email{jean-claude.ndogmo@wits.ac.za}

\author[FM Fazal]{F.M. Mahomed}

\address[FM Fazal]{Centre for Differential Equations, Continuum Mechanics and Applications, School of Computational and Applied Mathematics,
University of the
Witwatersrand,
Private Bag 3, Wits 2050,
South Africa}

\email{Fazal.Mahomed@wits.ac.za}

\begin{abstract}
An expression for the coefficients of a linear iterative equation in terms of the parameters of the source equation is given both for equations in standard form and for equations in reduced normal  form. The operator generating an iterative equation of a general order in reduced normal form is also obtained and some other properties of iterative equations are established. In particular, a simple necessary and sufficient condition for an equation to be iterative is given for the general fourth-order linear equation solely in terms of its coefficients.
\end{abstract}



\maketitle

\section{Introduction}
\label{s:intro}
It is well-known \cite{KMichel1} that linear ordinary differential equations ({\sc lode}s) of order one or two can all be reduced by a local diffeomorphism  of the $(x,y)$-plane to the canonical form $y'=0$ and $y''=0$, respectively, and that this is not the case for equations of a general order $n>2.$  Lie \cite{liep213} showed that a differential equation  of a general order $n>2$ is equivalent (by a local diffeomorphism of the plane) to the  equation $y^{(n)}=0,$ which we shall henceforth refer to as the canonical form of the linear equation, only if its symmetry algebra has the maximal dimension $n+4$. In a much recent paper, Krausse and Michel \cite{KMichel2} proved the converse of this statement and also showed that a {\sc lode} of order $n >2$ has a symmetry algebra of maximal dimension if and only if it is iterative.  Linear iterative equations are the iterations $\Psi^n y=0$ of a linear first order equation,  of the form

\begin{equation}\label{eq:iter1}
\Psi y\equiv r(x)y'+ s(x) y=0, \qquad \Psi^n y = \Psi^{(n-1)} \Psi y.
\end{equation}
In these iterations, the equation $r(x)y'+ s(x)y=0$  is termed the source equation, while the functions $r(x)$ and $s(x)$ are referred to as its parameters. Despite the unique symmetry properties of these iterative equations, many of their properties are still not known, and Mahomed \cite{mahom1} gave a listing of these equations for the orders three to five. We extend this list to equations of a general order, in both the standard form and the associated reduced normal form, and determine the operator generating the linear iterative equation of any given order in reduced normal form. Other properties of iterative equations are also established and in particular a simple criterion for a fourth-order linear equation to be iterative is given solely in terms of its coefficients.

\section{Equations in the general linear form}
By replacing the dependent variable $y=y(x)$ by $y + y_p,$ where $y_p$ is a particular solution of the inhomogeneous equation, we may assume without loss of generality that a linear iterative equation of a general order $n$ has the form
\begin{equation}
\label{eq:geniter}
\Psi^n y \equiv K_n^0\, y^{(n)} + K_{n}^1\, y^{(n-1)} +  K_{n}^2\, y^{(n-2)} + \dots +  K_{n}^{n-1}\, y' +  K_{n}^n\, y =0.
\end{equation}
%
\subsection{Case where $\Psi= \frac{d}{dx} + s$}
For the sake of clarity we first consider the case where the differential operator $\Psi= r d/ dx + s$ is much simpler, with $r\equiv r(x)=1.$ It is clear that in this case, the operator $\Psi$ leaves invariant the leading coefficient, and thus we have $K_n^0=1$ in this case. Using the formula for $\Psi^n y$ in \eqref{eq:iter1} gives the recurrence relations
\begin{subequations}\label{eq:rec4r=1}
\begin{align}
K_n^1 &= K_{n-1}^1 + s. \\
K_n^j &=  K_{n-1}^j + \frac{d}{dx} \left( K_{n-1}^{j-1}\right) + s K_{n-1}^{j-1}= K_{n-1}^j+ \Psi K_{n-1}^{j-1},\\
\intertext{ $\qquad \qquad \qquad \quad$ (for $j=2, \dots, n-1$ and $n >2$).}
K_n^n &= \frac{d}{dx} \left( K_{n-1}^{n-1}\right) + s K_{n-1}^{n-1}=\Psi K_{n-1}^{n-1}.
\end{align}
\end{subequations}
Setting
\begin{equation}\label{eq:knj=0}
K_m^j=0, \quad \text{for $j<0$ or $j > m$}, \text{ and }  K_m^j=1, \quad \text{for $m=j=0$}
\end{equation}
reduces the recurrence equations \eqref{eq:rec4r=1} to the single equation
\begin{equation} \label{eq:rec4r=1v2}
K_n^j = K_{n-1}^j + \Psi K_{n-1}^{j-1}, \qquad 0 \leq j \leq n, \quad \forall n \geq1.
\end{equation}
Solving the recurrence relations \eqref{eq:rec4r=1v2} together with the initial conditions

\begin{align*}
\Psi y   &\equiv y' + s y = y' + K_1^1 y  \\
\Psi^2 y &\equiv y'' + 2 s y' + (s^2+ s')y = y''+ K_2^1 y' + K_2^2 y
\end{align*}
readily gives
\begin{align*}
K_n^1 &= n s = \binom{n}{1} \Psi^0 s \\
K_n^n &= \Psi^{n-1} s,
\end{align*}

for all $n \geq 1.$ Using these two equalities, and setting
\begin{equation} \label{eq:psim1}
\Psi^{-1} f=1, \qquad \text{for every function $f=f(x)$},
\end{equation}
one readily sees by induction on $j$ and $n$ that
$$K_n^j = \binom{n}{j} \Psi^{j-1} s, \text{ for $j=0, \dots, n$ and $n \geq 1$}.$$

We have thus obtained the following result for the case where $r=1$.
\begin{thm}\label{th:casr=1}
When the operator $\Psi$ generating the iterative equation has the form $\Psi= d/dx + s$, that is, when it depends on the single function $s,$ the coefficients $K_n^j$  of the iterative equation of a general order $n$ are given by
\begin{equation} \label{eq:knj-r1}
K_n^j = \binom{n}{j} \Psi ^{j-1}s, \qquad \forall j=1, \dots, n,
\end{equation}
and the iterative equation of a general order $n$ is therefore given by
\begin{equation}
\Psi^n y = y^{(n)}+ \sum_{j=1}^n \left[ \binom{n}{j} \Psi^{j-1} s \right] y^{(n-j)}.
\end{equation}
\end{thm}
Note however that \eqref{eq:knj-r1} is also valid for $j=0.$
\subsection{Case where the source equation depends on both parameters}

In this case we have $\Psi= r d/dx +s,$ where the parameters $r,$ and $s$  are given functions,
and this is the most general case. Using the definition for $\Psi^n \,y$ given in \eqref{eq:geniter}
as well as the conventions for $K_n^j$ set in \eqref{eq:knj=0} show that the coefficients $K_n^j$ of
the iterative equation \eqref{eq:geniter} satisfy the recurrence relations

\begin{equation} \label{eq:recgen}
K_n^j= r K_{n-1}^{j} + \Psi K_{n-1}^{j-1}, \qquad \text{for $0 \leq j \leq n,$ and $n \geq 1$},
\end{equation}
which naturally reduce to \eqref{eq:rec4r=1v2} for $r=1.$  Setting $j=0$ or $j=n$ in \eqref{eq:recgen} readily gives by induction on $n$ the identities
\begin{equation}\label{eq:kn0knn}
K_{n}^0= r^n, \qquad K_n^n = \Psi^{n-1} s, \qquad \text{for all $n \geq 1$}.
\end{equation}
Applying  \eqref{eq:recgen} recursively and using the conventions set in \eqref{eq:knj=0} give a new recurrence relation
\begin{equation} \label{eq:recknj1}
K_n^j = \sum_{k=j}^n r^{n-k} \Psi K_{k-1}^{j-1}, \quad \text{for $j=0,\dots, n$ and $n \geq1.$}
\end{equation}

Although Eq. \eqref{eq:recknj1} does not provide the required expression for $K_n^j$ in terms
of the parameters $r$ and $s,$ it represents an algorithm for the computation of
the coefficients $K_n^j$ for all possible values of $n$ and $j$. For instance, using \eqref{eq:recknj1}
with $j = 1, 2$ readily gives

\begin{subequations}\label{eq:kn12}
\begin{align}
K_{n}^1 &= r^{n-1} \left[ n s + \binom{n}{2} r' \right] \label{eq:kn12a}  \\
K_n^2  &=  r^{n-2} \left[ \binom{n}{2} \Psi s + \binom{n}{3} \left( 3 s r' + r r''
    + \frac{3n-5}{4} r'^2 \right) \right]. \label{eq:kn12b}
\end{align}
\end{subequations}
%
To obtain the general expression for $K_n^j$ using a recurrence relation relating them, we rewrite \eqref{eq:recknj1} in the form
\begin{equation} \label{eq:recknj2}
K_n^j= \sum_{k=j}^n r^{(k-j)} \Psi K_{n-k+j-1}^{j-1}.
\end{equation}
Then, using \eqref{eq:knj=0} and \eqref{eq:kn0knn}, the following formulas are successively obtained:

\begin{subequations}
\begin{align}
K_n^1 &= \sum_{k=1}^n r^{k-1} \Psi r^{n-k}\label{eq:kn1gnl}\\
K_n^2 &= \sum_{k_2=2}^n \sum_{k_1=1}^{n-k_2+1} r^{k_2-2}\Psi \left[  r^{k_1-1} \Psi r^{n+1- (k_1+ k_2)}\right]\label{eq:kn2gnl}\\
\begin{split}
K_n^3 &=  \sum_{k_3=3}^{n}\;  \sum_{k_2=2}^{n- k_3+2} \; \sum_{k_1=1}^{\; n- (k_2+ k_3)+3} \\
       & \qquad  r^{k_3-3} \Psi\left[ r^{k_2-2} \Psi \left[ r^{k_1-1} \Psi r ^{n+3- (k_1+k_2+k_3)} \right]\right].
\end{split}
\end{align}
\end{subequations}
Continuing this process with two more iterations by computing $K_n^4$ and $K_n^5,$ a clear pattern for the general coefficient $K_n^j$ emerges, and to write down this expression we introduce some notations. For $n \geq 1$ and $0\leq i\leq j \leq n,$ set
\begin{subequations}\label{eq:betaij}
\begin{align}
\beta_{ij}&= \sum_{u=i+1}^j k_u, \qquad \text{ where $k_u \in \Z$ }\\
M_i &\equiv M_i(j) = n+ \binom{j}{2} - \binom{i}{2}- \beta_{ij}\\
\alpha_j &= n+ \binom{j}{2} - \beta_{0j}=M_0.
\end{align}
\end{subequations}
We have the following result for the general case, where $\Psi= r d/dx +s.$
\begin{thm}\label{th:knj}
In terms of the parameters $r$ and $s$ of the source equation, the general coefficient $K_n^j$ of the iterative equation \eqref{eq:geniter} has the form
\begin{equation} \label{eq:knjgnl}
\begin{split}
K_n^j &= \sum_{k_j=j}^{M_j}\; \sum_{k_{j-1}=j-1}^{M_{j-1}} \dots \sum_{k_2=2}^{M_2} \; \sum_{k_1=1}^{M_1}\\
 & r^{k_j-j} \Psi \left[ r^{k_{j-1}-{(j-1)}} \Psi \left[\dots \Psi \left[r^{k_1-1} \Psi r ^{\alpha_j}\right] \dots \right] \right],
\end{split}
\end{equation}
for $n\geq 1,$ and $1\leq j \leq n,$ and where the expressions for $\beta_{ij}, M_i,$ and $\alpha_j$ are given by \eqref{eq:betaij}.
\end{thm}
For the sake of clarity it would be useful to verify explicitly that \eqref{eq:knjgnl} reduces indeed to \eqref{eq:knj-r1} for $r=1$.
\begin{pro}\label{pro:r=1}
Eq. \eqref{eq:knjgnl} reduces as expected to \eqref{eq:knj-r1} for $r=1$.
\end{pro}
\begin{proof}
When $r=1,$ the general term in the summation \eqref{eq:knjgnl} clearly reduces to $\Psi^{j} \cdot 1= \Psi^{j-1} s,$ and since this expression does not depend on the running indices $k_1, k_2, \dots k_j$ to prove the proposition, it suffices to show that the total number
\begin{equation} \label{eq:pnj}
P_{n,j}= \sum_{k_j=j}^{M_j}\; \sum_{k_{j-1}=j-1}^{M_{j-1}} \dots \sum_{k_2=2}^{M_2} \; \sum_{k_1=1}^{M_1} 1
\end{equation}
of terms in this summation is precisely $\binom{n}{j}.$ For $j=1,$ and $j=2,$ it clearly follows from \eqref{eq:kn1gnl} and \eqref{eq:kn2gnl} respectively that $P_{n,1}= \binom{n}{1},$ and $P_{n,2}= \binom{n}{2}.$ It also follows from \eqref{eq:kn0knn} that $P_{n,0}$ and $P_{n,n}$ also satisfy the required property for $n\geq 1.$ Let $n\geq 2$ and assume that $P_{v,j-1}=\binom{v}{j-1},$ for $1\leq v<n$ and $0 \leq j-1 \leq v <n.$ Then it follows from \eqref{eq:recknj2} and the linearity of $\Psi$ that $P_{n,j}= \sum_{k=j}^{n} \binom{n-k+j-1}{j-1}.$ Setting $n-k+j-1=q$ and $j-1=m$ gives
$$P_{n,j}= \sum_{q=m}^{n-1} \binom{q}{m}= \binom{(n-1)+1}{m+1}= \binom{n}{j},$$
and this completes the proof by induction of the required property for $P_{n,j}.$
\end{proof}

  Although Proposition \ref{pro:r=1} gives a verification of the  validity of the complicated formulas \eqref{eq:knjgnl} at least for the simpler case $r=1,$ these formulas can be slightly simplified, by a suitable change of variables. Indeed, if in \eqref{eq:knjgnl} we set
\begin{equation}\label{eq:kj2pj}
\begin{cases}
k_j-j = n- P_j, \quad &\text{for}\quad  j\geq 1 \\
k_i-i= P_{i+1}-P_i -1, \quad &\text{for} \quad i=1,\dots , j-1
\end{cases}
\end{equation}
This reduces the expression for $M_i$ in \eqref{eq:betaij} to
\begin{align}\label{eq:newMi}
M_i&= P_{i+1}-1
\end{align}
In particular $\alpha_j= M_o$ is reduced to
\begin{align}\label{eq:newalphaj}
\alpha_j &= P_1-1
\end{align}
Since $M_j=n$ for all $j,$ thanks to \eqref{eq:newMi}, we may rewrite \eqref{eq:kj2pj} as
\begin{align}\label{eq:gnlkj2pj}
k_i-i &= M_i -P_i, \quad \text{for} \quad i=1,\dots,j.
\end{align}
Consequently, thanks to \eqref{eq:newMi}, \eqref{eq:newalphaj} and \eqref{eq:gnlkj2pj}, and after  renaming  the $P_i,$ we may rewrite \eqref{eq:knjgnl} in the slightly simplified form
\begin{equation} \label{eq:simplknj}
\begin{split}
K_n^j &= \sum_{k_j=j}^{n}\; \sum_{k_{j-1}=j-1}^{k_j-1} \dots \sum_{k_2=2}^{k_3-1} \; \sum_{k_1=1}^{k_2-1}\\
 & r^{n-k_j} \Psi \left[ r^{k_j-k_{j-1}-1}\Psi \left[\dots r^{k_3-k_2-1} \Psi \left[r^{k_2-k_1-1} \Psi r ^{k_1-1}\right] \dots \right] \right],
\end{split}
\end{equation}

\section{Equations in reduced normal form}

By diving through the general $n$-th order linear iterative equation $\Psi^n y $ in \eqref{eq:geniter} by $K_n^0= r^n$, it can be put in the form
\begin{equation} \label{eq:iterstd}
y^{(n)} + B_n^1\, y^{(n-1)}+ \dots + B_n^j\, y^{(n-j)} + \dots + B_n^n\, y=0,
\end{equation}
where $B_n^j= K_n^j /r^n.$ This is the standard form of the general linear iterative equation with leading coefficient one.  It is well-known that \eqref{eq:iterstd} can be transformed to the normal form (in which coefficient of $y^{(n-1)}$ has vanished) by a change of the dependent variable of the form
\begin{equation}\label{eq:std2nor}
y \mapsto y \exp\left(\frac{1}{n} \int_{x_0}^x B_n^1 (v)\, dv \right).
\end{equation}

However,  this amounts to the requirement that $B_n^1=0,$ i.e. that $K_n^1=0.$  Therefore, an $n$-th order linear equation in reduced normal form is iterative if and only if it has the form

\begin{subequations} \label{eq:iternor}
\begin{align}
&y^{(n)} + A_n^2\, y^{(n-2)} + \dots + A_n^j\, y^{(n-j)} + \dots A_n^n\, y=0  \label{eq:iternor1}\\
\intertext{where}\vspace{-10mm}
& A_n^j= \frac{K_n^j}{r^n}\Bigg \vert_{K_n^1=0} , \qquad (2 \leq j \leq n), \label{eq:iternor2}
\end{align}
\end{subequations}
and where $K_n^j$ is given by \eqref{eq:knjgnl}. It follows from \eqref{eq:kn12a} that setting $K_n^1=0$ amounts to setting
\begin{equation} \label{eq:s(x)=}
s = - \frac{1}{2} (n-1) r',
\end{equation}

and this shows why any iterative equation in normal form can be expressed
in terms of the parameter $r$ alone. Moreover, thanks to \eqref{eq:iternor2} the coefficients
 $A_n^j$ inherit all of the characterization \eqref{eq:knjgnl} obtained for the coefficients $K_n^j$ of
iterative equations in standard form, and the corresponding expression for
the $A_n^j$ has up to the factor $1/r^n,$ exactly the same form as that for the  $K_n^j$
because the parameter $s$ does not appear explicitly in Eq. \eqref{eq:knjgnl}. In addition,
using the same algorithm (e.g. \eqref{eq:recknj1}) obtained for the $K_n^j,$ one can readily
compute $A_n^j$ for all $n \geq 2$ and $2 \leq j \leq n.$ For instance, using the expression
for $K_n^2$ in \eqref{eq:kn12b} together with \eqref{eq:iternor2}, one readily sees that

\begin{equation} \label{eq:An2}
A_n^2= \binom{n+1}{3} A(r), \text{ where } A(r)\equiv A_2^2= \frac{r'^{\,2}- 2 r r''}{4 r^2}.
\end{equation}
\begin{thm}\label{th:Phi_n}
Let the differential operator $\Phi_n$ be given by
\begin{equation} \label{eq:Phi_n}
\Phi_n = \frac{1}{r^n} \Psi^n \Bigg \vert_{K_n^1 =0}.
\end{equation}
Then the equation $\Phi_n y=0$ is exactly \eqref{eq:iternor1}, that is $\Phi_n$ generates the (most general) linear iterative equation of an arbitrary order $n$ in normal form.
\end{thm}

\begin{proof}
Indeed, $\Psi^n y$ generates the linear iterative equation of general order $n$ in standard form and $(1/r^n) \Psi^n$ generates the same equation with leading coefficient 1, while setting $K_n^1=0$ corresponds as already noted to reduce the latter equation to its normal form.
\end{proof}

\section{Applications}

   Although formula \eqref{eq:knjgnl} in Theorem \ref{th:knj} gives a characterization of the coefficients of a linear iterative equation of a general order in terms of the parameters of the source equation, and therefore a characterization of the linear iterative equation itself, in practice a linear equation is given solely in terms of its coefficients, and without reference to any source equation or the parameters thereof. We thus need a characterization of these iterative equations that relies solely on the coefficients of the equation. Using the operator $\Phi_n$ of Theorem \ref{th:Phi_n} we can easily generate an iterative equation of a general order $n,$ and we let the generated equation be in the form \eqref{eq:iternor1}. By analyzing the coefficients of the latter equation we readily see, at least for low order equations, that they all depend only on the coefficient $A_n^2$ and its derivatives. For instance, for $n=3$ or $4,$ if we set $A_3^2= a_3$ and $A_4^2= a_4,$ then iterative equations of orders $3$ and $4$ take on respectively the forms

\begin{align}
&y''' + a_3 y' + \frac{1}{2} a_3' y =0  \label{eq:itero3}\\
&y^{(4)} + a_4 y'' + a_4' y' + \left(\frac{3}{10} a_4'' + \frac{9}{100}a_4^2\right)y =0 \label{eq:itero4}
\end{align}
Note that conversely, any equation of the form \eqref{eq:itero3} or \eqref{eq:itero4} is iterative, because we can always solve for $r$ the equation $A_n^2= \binom{n+1}{3}A(r)$ appearing in \eqref{eq:An2}, together with \eqref{eq:s(x)=} to find the parameters of the corresponding source equations. Thus each of the equations \eqref{eq:itero3} and \eqref{eq:itero4} characterizes the normal form of the iterative equation of the corresponding order. However, the most general characterization is to be given for equations in standard form \eqref{eq:iterstd}, and the most practical characterization should be expressed explicitly in terms of the coefficients of the equation.\par

  Let a third-order linear {\sc lode} be given in the form

\begin{equation}\label{eq:3odestd}
y'''+ c_2 y'' + c_1 y' + c_0 y=0,
\end{equation}

where the coefficients $c_j$ for $j=0,1,2$ are all functions of $x.$ We may thus assume that its reduced normal form is given by \eqref{eq:itero3}. Reverting back the latter reduced equation to the corresponding equation in standard form using the inverse of the transformation of type  \eqref{eq:std2nor} , and dropping the subscripts in the expression of the coefficients in \eqref{eq:itero3} gives
\begin{equation}\label{eq:2std33}
\begin{split}
 &w'''+ c_2 w''  + \left(a+ c_2'+ c_2^2/3\right) w' \\
 &+\frac{1}{54}  \left(27 a'+18 c_2''+18 a c_2+18
c_2' c_2+2 c_2^3\right)w =0,
\end{split}
\end{equation}
where $w$ is the new depend variable. Letting $w=y$ in \eqref{eq:2std33} and then equating its coefficients with those of \eqref{eq:3odestd} shows that

\begin{subequations}\label{eq:coford3}
\begin{align}
a &=  c_1 - (c_2'+ c_2^2/3) \label{eq:coford3a}\\
c_0 &= \frac{1}{54}  \left[27 a'+18 c_2''+18 a c_2+18 c_2' c_2+2 c_2^3 \right] \label{eq:coford3b}.
\end{align}
\end{subequations}
Using \eqref{eq:coford3a} for the expression of $a$ and its derivatives and substituting the result in \eqref{eq:coford3b} gives

\begin{equation}\label{eq:characof3}
54 c_0 - 18 c_1 c_2 + 4 c_2^3-27 c_1' + 18 c_2 c_2'+ 9 c_2'' =0.
\end{equation}
The latter equation is the well-known characterization  due to Laguerre \cite{lag} and Lie \cite{lieTransf} of linear third order equations that can be reduced to the canonical form $y'''=0$ by a point transformation. However, the methods they used to derive this conditions were different from the one used here, and which is based on the characterization of iterative equations given by \eqref{eq:knjgnl} and the generating operator $\Phi_n$ in \eqref{eq:Phi_n}. Clearly, due to the result already cited of Krause and Michel \cite{KMichel2} relating iterative equations and equations reducible to canonical form by point transformations, \eqref{eq:characof3} is also a characterization of linear third order equations which are iterative.

We now move on to consider the case of a fourth order linear equation given in the form
\begin{equation}\label{eq:4odestd}
y^{(4)}+ c_3y''' + c_2 y'' + c_1 y' + c_0 y=0,
\end{equation}
and we are interested in deriving necessary and sufficient conditions on the coefficients $c_j\; (j=0, \dots, 3)$  for the equation to be iterative. Here again, we may assume that the reduced normal form of the equation is given by \eqref{eq:itero4}. Transforming back the latter equation to its standard form, dropping the subscript in $a_4$ and  equating the coefficients of the resulting equation with those of \eqref{eq:4odestd} gives
\begin{subequations} \label{eq:coford4}
\begin{align}
a &= \frac{1}{8} \left(   -3 c_3^2+ 8 c_2 - 12 c_3' \right) \label{eq:coford4a}\\
c_1&=  \frac{1}{2} a c_3 + \frac{1}{16}c_3^3+ a' + \frac{3}{4} c_3 c_3' + c_3'' \label{eq:coford4b}\\
\begin{split}
6400 c_0 &=  576 a^2+400 a \left(4 c_3'+c_3^2\right) +80
\left(15 c_3'^2 +24 a''+20 c_3''' \right)\\
&\quad+1600 \left(a'+c_3'' \right) c_3+ 600 c_3'
c_3^2+25 c_3^4.  \label{eq:coford4c}
\end{split}
\end{align}
\end{subequations}
Substituting in \eqref{eq:coford4b} and \eqref{eq:coford4c} the expressions for $a$ and its derivatives given by \eqref{eq:coford4a} yields
\begin{subequations}\label{eq:characof4}
\begin{align}
0&=4 c_2 c_3 - c_3^3 + 8 c_2' - 6 c_3 c_3' - 4 c_3'' - 8 c_1  \label{eq:characof4a}\\
\comment{Here is where correction takes place, below!!!}
\begin{split}
0&= 1600 c_0 - 144 c_2^2 + 11 c_3^4 -400 c_3 c_2' + 288 c_3^2 c_3'+ 336 c_3^{\,\prime\, 2} \\
&\quad + 8 c_2(c_3^2 + 4 c_3')- 480 c_2'' + 560 c_3c_3'' + 320 c_3'''. \label{eq:characof4b}
\end{split}
\end{align}
\end{subequations}
We thus have the following result.
\begin{thm}
A Linear fourth order equation of the general form \eqref{eq:4odestd} is iterative if and only if its coefficients satisfy the system of two equations \eqref{eq:characof4}.
\end{thm}
\begin{proof}
As the system \eqref{eq:characof4} was obtained under the assumption that \eqref{eq:4odestd} is iterative, we only need to prove conversely that the equation is iterative whenever its coefficients satisfy \eqref{eq:characof4}. The reduced normal form of \eqref{eq:4odestd} has, after the substitution of the expressions for $c_0$ and $c_1$ given by \eqref{eq:characof4} in terms of $c_2, c_3,$ and their derivatives, the form

\begin{subequations}\label{eq:nor4std4}
\begin{align}
w^{(4)} & + Q_2 w'' + Q_1 w' + Q_0 w  =0 \label{eq:nor4std4E} \\
\intertext{where}
Q_2 &= c_2 - \frac{3}{8}(c_3^2 + 4 c_3')\\
Q_1 &= c_2'- \frac{3}{4}(c_3 c_3'+ 2 c_3'')\\
 \begin{split}  Q_0 &= \frac{3}{6400}(192 c_2^2 + 27 c_3^4 - 48 c_3'^2 - 144 c_2 (c_3^2+4 c_3'))\\
&\quad + \frac{3}{6400}( 27 c_3^4+ 216 c_3^2 c_3'+ 640 c_2''- 480 c_3 c_3'' -960 c_3'''). \end{split}
\end{align}
\end{subequations}
The coefficients $Q_j$ thus obtained clearly satisfy the conditions
$$Q_1= Q_2'\quad \text{ and }\quad  Q_0= (\frac{3}{10} Q_2'' + \frac{9}{100} Q_2^2)$$
prescribed by \eqref{eq:itero4}, and this completes the proof of the theorem.
\end{proof}


\section{Concluding remarks}

   Some of the results that we've obtained in this paper include the expression given in \eqref{eq:knjgnl} for  the coefficients of an iterative equation of a general order in standard form \eqref{eq:geniter} in terms of the parameters of the source equation, and the similar expression in \eqref{eq:iternor2} for iterative equations in reduced normal form \eqref{eq:iternor1}.  We have provided a simple algorithm \eqref{eq:recknj1} for the calculation of these coefficients for given parameters of the source equation, and the operator  $\Phi_n$ obtained in Theorem \ref{th:Phi_n} generates the iterative equation of an arbitrary order $n$ in normal form. Given the already cited result of Krause and Michel \cite{KMichel2} according to which an equation is iterative if and only if it is reducible by a point transformation to the canonical form, the characterization of iterative equation we have given in \eqref{eq:characof4} is also an extension to the order four of the characterization \eqref{eq:characof3} due to Laguerre \cite{lag} and Lie \cite{lieTransf} of linear third-order equations that are reducible to the canonical form. But unlike the characterization \eqref{eq:characof3} which consists of a single equation, the characterization \eqref{eq:characof4} for fourth-order equations consists of a system of two  equations, and from the process of their derivation it should be expected that for equations of order $n$ the corresponding characterization will consist of a system of $n-2$ equations.

\bibliographystyle{model1-num-names.bst}

\end{document}